\documentclass{amsart}

\usepackage{amssymb,amsmath,amsthm,latexsym,stmaryrd}
\usepackage{color}

\newtheorem{thm}{Theorem}[section]

\newtheorem*{theorem A}{Theorem A}
\newtheorem*{theorem B}{N\"olker's Theorem}

\theoremstyle{remark}

\theoremstyle{remark}

\theoremstyle{definition}

\numberwithin{equation}{section}
\def\({\left ( }
    \def\){\right )}
\def\<{\left < }
\def\>{\right >}

\newcommand{\ie}{i.e. }
\newcommand{\s}{\mathfrak{S}}

\newcommand{\n}{\nabla}

\newcommand{\z}{\mathfrak{z}}
\newcommand{\W}{\mathcal{W}}
\newcommand{\MM}{\mathcal{M}}
\newcommand{\ea}{\varepsilon_\alpha}
\newcommand{\ta}{\theta}

\newcommand{\gm}{\gamma}
\newcommand{\al}{\alpha}
\newcommand{\bt}{\beta}

\begin{document}

\vspace{2cm}

\title[Almost hypercomplex manifolds with Hermitian-Norden metrics \dots]
{Almost hypercomplex manifolds with Hermitian-Norden metrics and 4-dimensional indecomposable real Lie algebras depending on two parameters}

\author{Hristo Manev}
\address{Medical University -- Plovdiv, Faculty of Public Health,
Department of Medical Informatics, Biostatistics and e-Learning,   15-A Vasil Aprilov
Blvd.,   Plovdiv 4002,   Bulgaria;}
\email{hristo.manev@mu-plovdiv.bg}





\begin{abstract}
The object of investigations are almost hypercomplex structures with Hermitian-Norden metrics on 4-dimensional Lie groups considered as smooth manifolds. There are studied both the basic classes of a classification of 4-dimensional indecomposable real Lie algebras depending on two parameters. Some geometric characteristics of the respective almost hypercomplex manifolds with Hermitian-Norden metrics are obtained.\thanks{The author was supported by Project MU19-FMI-020 of the Scientific Research Fund, University of Plovdiv, Bulgaria and National Scientific Program ''Young Researchers and Post-Doctorants''}
\end{abstract}
\maketitle
\section*{Introduction}

The almost hypercomplex structure $H$ on a $4n$-dimensional manifold $\MM^{4n}$ is a triad of anticommuting almost complex structures
such that each of them is a composition of the two other structures. The introduced structure $H$ is equipped with a metric structure of Hermitian-Norden type,
generated by a pseudo-Riemannian metric $g$ of neutral signature
\cite{GriMan24,GriManDim12}. In the considered case, one of the almost complex structures of $H$ acts as an isometry with respect to $g$ in each tangent fibre and the other two act as anti-isometries.
Therefore, there exist three (0,2)-tensors
associated with $H$ to the metric $g$: a K\"ahler form and two metrics of the Hermitian-Norden type.
In the considered case, on the one hand $g$ is a Hermitian metric with respect to $J_1$ and on the other hand $g$ is a Norden metric regarding $J_2$ and $J_3$. The derived structure is called an almost hypercomplex structure with Hermitian-Norden metrics.

The geometry of the considered manifolds is investigated in \cite{GriMan24,GriManDim12,Ma05,Man28,ManGri32,ManSek,NakHMan}. This type of manifolds are the only possible way to involve Norden-type metrics on almost hypercomplex manifolds.
In \cite{Barb,Fino,HM5,ManTav2,Ovando,ZamNak}, there are studied similar structures and metrics on Lie groups considered as manifolds.

The present paper is organized as follows. In
Sect.~\ref{sect-prel} we recall
some facts about the almost hypercomplex manifold with Hermitian-Norden  metrics. %
In Sect.~\ref{sect-lie} we construct almost hypercomplex structure with Hermitian-Norden metrics on 4-dimensional Lie groups considered as smooth manifolds.
The last Sect.~\ref{sect-45} and Sect.~\ref{sect-46} are devoted to the study of some geometric characteristics of both of the basic classes of a classification of 4-dimensional indecomposable real Lie algebras depending on two parameters.

\section{Preliminaries}\label{sect-prel}

\subsection{Almost hypercomplex manifolds with Hermitian-Norden metrics}
Let $(\MM,H)$ be an almost hypercomplex manifold, \ie $\MM$ is a
$4n$-dimension\-al differentiable manifold and $H=(J_1,J_2,J_3)$
is an almost hypercomplex structures on $\MM$ with the following
properties for all cyclic permutations $(\al, \bt, \gm)$ of
$(1,2,3)$:%
\begin{equation*}\label{J123} %
J_\al=J_\bt\circ J_\gm=-J_\gm\circ J_\bt, \qquad
J_\al^2=-I,
\end{equation*} %
where $I$ denotes the identity. 

Let $g$ be a neutral metric on $(\MM,H)$ having the properties
\begin{equation}\label{gJJ} %
g(\cdot,\cdot)=\ea g(J_\al \cdot,J_\al \cdot),
\end{equation} %
where
\begin{equation*}\label{epsiloni}
\ea=
\begin{cases}
\begin{array}{ll}
1, \quad & \al=1;\\
-1, \quad & \al=2;3.
\end{array}
\end{cases}
\end{equation*}

Here and further, $\alpha$ runs over the range
$\{1,2,3\}$ unless otherwise is stated.

The associated 2-form $g_1$ and the
associated neutral metrics $g_2$ and $g_3$ are determined respectively by
\begin{equation}\label{gJ} %
g_\al(\cdot,\cdot)=g(J_\al \cdot,\cdot)=-\ea g(\cdot,J_\al \cdot).
\end{equation}%

The derived structure $(H,G)=(J_1,J_2,J_3;g,g_1,g_2,g_3)$ on
$\MM^{4n}$ is called an \emph{almost hy\-per\-com\-plex
structure with Hermit\-ian-Norden metrics} and the respective manifold $(\MM,H,G)$ is called an \emph{almost hypercomplex
manifold with Hermit\-ian-Norden metrics} (\cite{GriManDim12}). 

In \cite{GriManDim12} are introduced the fundamental tensors of a manifold $(\MM,H,G)$ by the following three
$(0,3)$-tensors
\begin{equation*}\label{F'-al}
F_\al (x,y,z)=g\bigl( \left( \n_x J_\al
\right)y,z\bigr)=\bigl(\n_x g_\al\bigr) \left( y,z \right),
\end{equation*}
where $\n$ is the Levi-Civita connection generated by $g$.
These tensors have the basic properties
\begin{equation}\label{FaJ-prop}
  F_{\al}(x,y,z)=-\ea F_{\al}(x,z,y)=-\ea F_{\al}(x,J_{\al}y,J_{\al}z)
\end{equation}
and the following relations are valid for all cyclic permutations $(\al, \bt, \gm)$ of
$(1,2,3)$
\begin{equation*}\label{F1F2F3}
    F_\al(x,y,z)=F_\bt(x,J_\gm y,z)-\ea F_\gm(x,y,J_\bt z).
\end{equation*}

The corresponding Lee forms $\ta_\al$ are defined by
\begin{equation*}\label{theta-al}
\ta_\al(\cdot)=g^{kl}F_\al(e_k,e_l,\cdot)
\end{equation*}%
for an arbitrary basis $\{e_1,e_2,\dots, e_{4n}\}$ of $T_p\MM$,
$p\in \MM$.


Let us remark that, according to \eqref{gJJ}, $(M,H,G)$ is an almost Hermitian manifold with respect to $J_1$ and it is an almost complex manifold with Norden metric with respect to $J_2$ and $J_3$.
The basic classes of the mentioned two types of
manifolds are given in \cite{GrHe} and \cite{GaBo}, respectively. The special class of the K\"ahler-type manifolds $\W_0(J_\al):$ $F_\al=0$  belongs to any other class in the corresponding classification.
In the lowest 4-dimensional case, the four basic classes
of almost Hermitian manifolds with respect to
$J_1$ are restricted to two:
the class of the almost K\"ahler manifolds $\W_2(J_1)$ and
the class of the Hermitian manifolds $\W_4(J_1)$ which, for dimension 4, are determined by:
\begin{equation}\label{cl-H-dim4}
\begin{split}
&\W_2(J_1):\; \mathop{\s}_{x,y,z}\bigl\{F_1(x,y,z)\bigr\}=0; \\
&\W_4(J_1):\; F_1(x,y,z)=\dfrac{1}{2}
                \left\{g(x,y)\ta_1(z)-g(x,J_1y)\ta_1(J_1z)\right. \\
&\phantom{\W_4(J_1):\; F_1(x,y,z)=\quad\,}
                \left.-g(x,z)\ta_1(y)+g(x,J_1z)\ta_1(J_1y)
                \right\},
\end{split}
\end{equation}
where $\s $ is the cyclic sum by three
arguments $x$, $y$, $z$.
The basic classes of the almost complex manifolds with Norden metric ($\al=2$ or $3$)
are determined, for dimension $4$, as follows:
\begin{equation}\label{cl-N-dim4}
\begin{split}
&\W_1(J_\al):\; F_\al(x,y,z)=\dfrac{1}{4}\bigl\{
g(x,y)\ta_\al(z)+g(x,J_\al y)\ta_\al(J_\al z)\bigr.\\
&\phantom{\W_1(J_\al):\; F_\al(x,y,z)=\quad\,\,} %
\bigl.+g(x,z)\ta_\al(y)
    +g(x,J_\al z)\ta_\al(J_\al y)\bigr\};\\
&\W_2(J_\al):\; \mathop{\s}_{x,y,z}
\bigl\{F_\al(x,y,J_\al z)\bigr\}=0,\qquad \ta_\al=0;\\
&\W_3(J_\al):\; \mathop{\s}_{x,y,z} \bigl\{F_\al(x,y,z)\bigr\}=0.
\end{split}
\end{equation}

As it is well known, the Nijenhuis tensor in terms of the covariant derivatives of $J_\al$ is defined by
$
    N(x,y)=\left(\n_x J_\al\right)J_\al y-\left(\n_y J_\al\right)J_\al x+\left(\n_{J_\al x}
    J_\al\right)y-\left(\n_{J_\al y}
    J_\al\right)x.
$
The following properties for $N$ are valid:
\begin{equation}\label{N-prop}
    N(x,y)=-N(y,x)=-N(J_\al x,J_\al y).
\end{equation}
The corresponding (0,3)-tensor is determined as $N(x,y,z)=g\left(N(x,y),z\right)$.

Let $R=\left[\n,\n\right]-\n_{[\ ,\ ]}$ be the
curvature (1,3)-tensor of $\nabla$ and the corresponding curvature
$(0,4)$-tensor with respect to the metric $g$ be denoted by the same letter: $R(x,y,z,w)$
$=g(R(x,y)z,w)$. The following properties are valid:
\begin{equation}\label{R-prop}
\begin{array}{c}
    R(x,y,z,w)=-R(y,x,z,w)=-R(x,y,w,z), \\[4pt]
R(x,y,z,w)+R(y,z,x,w)+R(z,x,y,w)=0.
\end{array}
\end{equation}

The Ricci tensor $\rho$ and the scalar curvature $\tau$ for $R$ as well as
their associated quantities are defined as usually by:
\begin{equation*}
\begin{array}{c}
    \rho(y,z)=g^{ij}R(e_i,y,z,e_j),\qquad \rho^*(y,z)=g^{ij}R(e_i,y,z,J_\al e_j),\\[6pt]
    \tau=g^{ij}\rho(e_i,e_j), \qquad \tau^*=g^{ij}\rho^*(e_i,e_j),\qquad \tau^{**}=g^{ij}\rho^*(e_i,J_\al e_j).
\end{array}
\end{equation*}

Every non-degenerate 2-plane $\mu$ with a basis $\{x,y\}$ with respect to $g$ in
$T_pM$, $p \in M$, has the following sectional curvature
\begin{equation*}\label{sect}
k(\mu;p)=\frac{R(x,y,y,x)}{g(x,x)g(y,y)-g(x,y)^2}.
\end{equation*}
A 2-plane $\mu$ is said to be holomorphic (resp., totally real) if $\mu= J_\al \mu$ 
(resp., $\mu \perp J_\al \mu \neq \mu$) with respect to $g$ and $J_\al$.

\section{Four-dimensional indecomposable real Lie algebras and almost hy\-per\-com\-plex
structure with Hermit\-ian-Norden metrics}\label{sect-lie}
Let $L$ be a simply connected 4-dimensional real Lie group
admitting an invariant hypercomplex structure, \ie left translations by elements of $L$ are holomorphic with respect to
$J_{\al}$ for all $\al$. Let $\mathfrak{l}$ be the corresponding Lie algebra of $L$.

We define a standard hypercomplex structure on $\mathfrak{l}$ as in \cite{So}:
\begin{equation*}\label{JJJ}
\begin{array}{llll}
J_1e_1=e_2, \quad & J_1e_2=-e_1,  \quad &J_1e_3=-e_4, \quad &J_1e_4=e_3;
\\[6pt]
J_2e_1=e_3, &J_2e_2=e_4, &J_2e_3=-e_1, &J_2e_4=-e_2;
\\[6pt]
J_3e_1=-e_4, &J_3e_2=e_3, &J_3e_3=-e_2, &J_3e_4=e_1,\\[6pt]
\end{array}
\end{equation*}
where $\{e_1,e_2,e_3,e_4\}$ is a basis of a 4-dimensional real Lie
algebra $\mathfrak{l}$ with center $\z$ and derived Lie algebra
$\mathfrak{l}'=[\mathfrak{l},\mathfrak{l}]$.

We introduce a pseudo-Euclidian metric $g$ of
neutral signature by
\begin{equation*}\label{g}
g(x,y)=x^1y^1+x^2y^2-x^3y^3-x^4y^4,
\end{equation*}
where $x(x^1,x^2,x^3,x^4)$, $y(y^1,y^2,y^3,y^4) \in \mathfrak{l}$. The metric $g$
generates an almost hypercomplex structure with Hermitian-Norden metrics on $\mathfrak{l}$, according to \eqref{gJJ} and \eqref{gJ} and $(L,H,G)$ is an almost hypercomplex manifold with Hermitian-Norden metrics.

The real 4-dimensional indecomposable Lie algebras are classified for instance by Mubarakzyanov in \cite{Muba}.
They can be found easily in \cite{GhaTho}, where there are given the non-zero brackets in each class according a basis $\{e_1,e_2,e_3,e_4\}$. Now, we focus our investigations on the basic classes $g_{4,5}$ and $g_{4,6}$ of this classification which depend on two real parameters $a$ and $b$:
\begin{equation}\label{45}
g_{4,5}:\; [e_1,e_4]=e_1,  [e_2,e_4]=a e_2,  [e_3,e_4]=b e_3, \; (a\neq0, b\neq0)
\end{equation}
\begin{equation}\label{46}
g_{4,6}:\; [e_1,e_4]=a e_1,  [e_2,e_4]=b e_2-e_3,  [e_3,e_4]=e_2+b e_3, \; (a\neq0, b\geq0)
\end{equation}

Further, the indices $i$, $j$, $k$, $l$ run over the range $\{1,2,3,4\}$.

\section{The class $g_{4,5}$}\label{sect-45}
Let us consider a manifold $(L,G,H)$ with corresponding Lie algebra from $g_{4,5}$.

We compute the basic components $(F_{\al})_{ijk}=F_{\al}(e_i,e_j,e_k)$. The non-zero of them are the following and their equal ones by \eqref{FaJ-prop}
\begin{equation}\label{Fijk-45}
\begin{array}{c}
(F_{1})_{113}=(F_{2})_{112}=-\frac{1}{2}(F_{3})_{111}=1, \\[6pt]
(F_{1})_{214}=(F_{3})_{212}=-\frac{1}{2}(F_{2})_{222}=-a,\qquad
(F_{2})_{314}=(F_{3})_{313}=b.
\end{array}
\end{equation}
The only non-zero components $(\theta_{\al})_i=(\theta_{\al})(e_i)$
of the corresponding Lee forms are
\begin{equation}\label{ta-i-45}
\begin{array}{c}
(\ta_1)_3=a+1,\quad
(\ta_2)_2=2a+b+1,\quad
(\ta_3)_1=-(a+b+2).
\end{array}
\end{equation}

\begin{thm}\label{thm-45}
Let $(L,H,G)$ be an almost hypercomplex manifold with Hermitian-Norden metrics with Lie algebra from the class $g_{4,5}$. It belongs to a certain class regarding $J_\al$ given in the following table, where are distinguished the different cases concerning the parameters $a$ and $b$.\\
\begin{center}
\begin{tabular}{|c|c|c|c|c|}
  \hline
  $a$ & $b$ & $J_1$ & $J_2$ & $J_3$ \\ \hline
  $-1$ & $1$ & $\W_2$ & $\W_2$ & $\W_{1}\oplus\W_{2}\oplus\W_{3}$ \\
  $-1$ & $-1$ & $\W_2$ & $\W_{1}\oplus\W_{2}\oplus\W_{3}$ & $\W_2$ \\
  $-1$ & $\neq \pm 1$ & $\W_2$ & $\W_{1}\oplus\W_{2}\oplus\W_{3}$ & $\W_{1}\oplus\W_{2}\oplus\W_{3}$ \\
  $1$ & $1$ & $\W_4$ & $\W_1$ & $\W_{1}\oplus\W_{2}$ \\
  $1$ & $-3$ & $\W_4$ & $\W_{2}\oplus\W_{3}$ & $\W_{2}\oplus\W_{3}$ \\
  $1$ & $\neq1$; $\neq-3$ & $\W_4$ & $\W_{1}\oplus\W_{2}\oplus\W_{3}$ & $\W_{1}\oplus\W_{2}\oplus\W_{3}$ \\
  $\neq \pm 1$ & $1$ & $\W_{2}\oplus\W_{4}$ & $\W_{1}\oplus\W_{2}$ & $\W_{1}\oplus\W_{2}\oplus\W_{3}$ \\
  $-\frac{1}{3}$ & $-\frac{1}{3}$ & $\W_{2}\oplus\W_{4}$ & $\W_{2}\oplus\W_{3}$ & $\W_{1}\oplus\W_{2}$ \\
  $-\frac{1}{2}(b+1)$ & $\neq -3$; $\neq -\frac{1}{3}$ & $\W_{2}\oplus\W_{4}$ & $\W_{2}\oplus\W_{3}$ & $\W_{1}\oplus\W_{2}\oplus\W_{3}$ \\
  $b$ & $\neq \pm 1$; $\neq -\frac{1}{3}$ & $\W_{2}\oplus\W_{4}$ & $\W_{1}\oplus\W_{2}\oplus\W_{3}$ & $\W_{1}\oplus\W_{2}$ \\
  $-b-2$ & $\neq -1$; $\neq -3$ & $\W_{2}\oplus\W_{4}$ & $\W_{1}\oplus\W_{2}\oplus\W_{3}$ & $\W_{2}\oplus\W_{3}$ \\
  $\neq 0$ & $\neq 0$ & $\W_{2}\oplus\W_{4}$ & $\W_{1}\oplus\W_{2}\oplus\W_{3}$ & $\W_{1}\oplus\W_{2}\oplus\W_{3}$ \\
  \hline
\end{tabular}
\end{center}
\end{thm}

\begin{proof}
Using the results in \eqref{Fijk-45}, \eqref{ta-i-45} and the classification conditions
\eqref{cl-H-dim4}, \eqref{cl-N-dim4} for dimension 4, we proof the assertion in each case.
\end{proof}

The non-zero basic components $(N_{\al})_{ijk}=N_{\al}(e_i,e_j,e_k)$ of the Nijenhuis tensor $N$ for $J_\al$ are
\begin{equation}\label{N-45}
\begin{array}{c}
(N_{1})_{132}=(N_{1})_{231}=1-a, \quad (N_{2})_{123}=(N_{2})_{231}=1-b,\\[6pt] (N_{3})_{123}=(N_{3})_{132}=a-b
\end{array}
\end{equation}
and the rest of them are calculated using \eqref{N-prop}.

There are computed the basic components $R_{ijkl}=R(e_i,e_j,e_k,e_l)$ of
the curvature tensor $R$. The non-zero of them are the following and their equal ones by \eqref{R-prop}
\begin{equation}\label{R-45}
\begin{array}{c}
R_{1221}=a, \quad R_{1313}=b, \quad R_{1414}=1, \\[6pt]
R_{2323}=ab, \quad R_{2424}=a^2, \quad R_{3443}=b^2.
\end{array}
\end{equation}

We obtain the basic components $\rho_{jk}=\rho(e_j,e_k)$, $(\rho_\al^*)_{jk}=\rho_\al^*(e_j,e_k)$, the values of $\tau$, $\tau_\al^*$, $\tau_\al^{**}$ and $k_{ij}=k(e_i,e_j)$:
\begin{equation}\label{rho-tau-45}
\begin{array}{c}
\rho_{11}=a+b+1, \quad \rho_{22}=a(a+b+1), \\[6pt]
\rho_{33}=-b(a+b+1), \quad \rho_{44}=-(a^2+b^2+1),  \\[6pt]
(\rho_1^*)_{12}=a, \quad (\rho_1^*)_{34}=b^2, \quad (\rho_2^*)_{13}=b, \\[6pt] (\rho_2^*)_{24}=a^2, \quad (\rho_3^*)_{14}=-1, \quad (\rho_3^*)_{23}=ab,\\[6pt]
\tau=2(a^2+b^2+ab+a+b+1), \quad \tau_1^*=\tau_2^*=\tau_3^*=0, \\[6pt] \tau_1^{**}=2(a+b^2), \quad \tau_2^{**}=2(a^2+b), \quad \tau_3^{**}=2(ab+1), \\[6pt]
k_{12}=a, \quad k_{13}=b, \quad k_{14}=1, \\[6pt]
k_{23}=ab, \quad k_{24}=a^2, \quad k_{34}=b^2.
\end{array}
\end{equation}


\begin{thm}\label{thm-45-2}
Let $(L,H,G)$ be an almost hypercomplex manifold with Hermitian-Norden metrics with Lie algebra from the class $g_{4,5}$. The following characteristics are valid:
\begin{enumerate}
  \item An $(L,H,G)$ has vanishing Nijenhuis tensor if and only if $a=b=1$;
  \item Every $(L,H,G)$ is non-flat;
  \item Every $(L,H,G)$ has a positive scalar curvature;
  \item Every $(L,H,G)$ is $*$-scalar flat;
  \item An $(L,H,G)$ is $**$-scalar flat w.r.t. $J_1$ if and only if $a=-b^2$;
  \item An $(L,H,G)$ is $**$-scalar flat w.r.t. $J_2$ if and only if $b=-a^2$;
  \item An $(L,H,G)$ is $**$-scalar flat w.r.t. $J_3$ if and only if $ab=-1$;
  \item An $(L,H,G)$ has a positive basic holomorphic sectional curvatures w.r.t. $J_1$ if and only if $a>0$;
  \item An $(L,H,G)$ has a positive basic holomorphic sectional curvatures w.r.t. $J_2$ if and only if $b>0$;
  \item An $(L,H,G)$ has a positive basic holomorphic sectional curvatures w.r.t. $J_3$ if and only if $ab>0$;
  \item An $(L,H,G)$ has a positive basic totally real sectional curvatures if and only if $a>0$ and $b>0$.
\end{enumerate}
\end{thm}
\begin{proof}
By virtue of \eqref{N-45}, \eqref{R-45} and \eqref{rho-tau-45}, we establish the truthfullness of the statement.
\end{proof}

\section{The class $g_{4,6}$}\label{sect-46}
In this section we focus our investigations on a manifold $(L,G,H)$ with corresponding Lie algebra from $g_{4,6}$.

By similar way as the previous section, we obtain the following results:
\begin{equation}\label{F-ta-46}
\begin{array}{c}
(F_{1})_{323}=(F_{2})_{223}=-(F_{3})_{213}=(F_{3})_{334}=\frac{1}{2}(F_{2})_{322}=1\\[6pt]
(F_{1})_{223}=(F_{2})_{314}=(F_{3})_{234}=(F_{3})_{313}=\frac{1}{2}(F_{2})_{222}=b,\\[6pt]
(F_{1})_{113}=(F_{2})_{112}=-\frac{1}{2}(F_{3})_{111}=a,\\[6pt]
(\ta_1)_2=(\ta_2)_3=-\frac{1}{2}(\ta_3)_4=1,\\[6pt]
(\ta_1)_3=-\frac{1}{2}(\ta_3)_1=a+b, \quad (\ta_2)_2=a+3b, \\[6pt]
\end{array}
\end{equation}
\begin{equation}\label{res-46}
\begin{array}{c}
(N_{1})_{132}=(N_{1})_{231}=(N_{2})_{123}=(N_{2})_{231}=a-b, \\[6pt]
(N_{1})_{133}=(N_{1})_{234}=-(N_{2})_{122}=-(N_{2})_{144}=-1,\\[6pt]
R_{1221}=R_{1313}=ab, \quad R_{1231}=a, \quad R_{1414}=a^2, \quad R_{2323}=b^2+1,\\[6pt]
R_{2424}=b^2-1, \quad R_{2434}=2b, \quad R_{3443}=2-b^2, \\[6pt]
\rho_{11}=a(a+2b), \quad \rho_{22}=-\rho_{33}=b(a+2b), \\[6pt] \rho_{23}=a+2b, \quad \rho_{44}=-(a^2+2b^2-2),\\[6pt]
(\rho_1^*)_{13}=-(\rho_2^*)_{12}=\frac{1}{2}(\rho_3^*)_{11}=a, \quad
(\rho_1^*)_{24}=-(\rho_2^*)_{34}=\frac{1}{2}(\rho_3^*)_{44}=-2b,\\[6pt]
(\rho_1^*)_{12}=(\rho_2^*)_{13}=ab, \quad (\rho_1^*)_{34}=(\rho_2^*)_{24}=b^2-1, \\[6pt] (\rho_3^*)_{23}=b^2+1, \quad (\rho_3^*)_{14}=-a^2,\\[6pt]
\tau=2(a^2+3b^2+2ab-1), \quad \tau_1^*=\tau_2^*=0, \quad \tau_3^*=2(a+2b), \\[6pt]
\tau_1^{**}=\tau_2^{**}=2(b^2+ab-1), \quad \tau_3^{**}=2(a^2+b^2+1), \\[6pt]
k_{12}=k_{13}=ab, \quad k_{14}=a^2, \quad k_{23}=b^2+1, \quad k_{24}=k_{34}=b^2-1
\end{array}
\end{equation}
and the rest nonzero of them we calculate using \eqref{FaJ-prop}, \eqref{N-prop} and \eqref{R-prop}.

\begin{thm}\label{thm-45}
Let $(L,H,G)$ be an almost hypercomplex manifold with Hermitian-Norden metrics with Lie algebra from the class $g_{4,6}$. Then $(L,H,G)$ belongs to the class
\[
(\W_2\oplus\W_4)(J_1) \cap (\W_1\oplus\W_2\oplus\W_3)(J_2) \cap (\W_1\oplus\W_2)(J_3)
\]
for each $a\neq0$, $b\geq0$. Moreover, $(L,H,G)$ does not belong to neither of $\W_2$, $\W_4$ for $J_1$; $\W_1$, $\W_2$, $\W_3$, $\W_1\oplus\W_2$, $\W_1\oplus\W_3$, $\W_2\oplus\W_3$ for $J_2$; $\W_1$, $\W_2$ for $J_3$.
\end{thm}
\begin{proof}
Using the results in \eqref{F-ta-46} and the classification conditions
\eqref{cl-H-dim4}, \eqref{cl-N-dim4} for dimension 4, we proof the assertions.
\end{proof}

\begin{thm}\label{thm-46-2}
Let $(L,H,G)$ be an almost hypercomplex manifold with Hermitian-Norden metrics with Lie algebra from the class $g_{4,6}$. The following characteristics are valid:
\begin{enumerate}
  \item Every $(L,H,G)$ has vanishing Nijenhuis tensor w.r.t. $J_3$;
  \item Every $(L,H,G)$ is non-flat;
  \item An $(L,H,G)$ is scalar flat if and only if $a=-b \pm \sqrt{1-2b^2}$, $(b\in[0;\frac{\sqrt{2}}{2}])$;
  \item Every $(L,H,G)$ is $*$-scalar flat w.r.t. $J_1$ and $J_2$;
  \item An $(L,H,G)$ is $*$-scalar flat w.r.t. $J_3$ if and only if $a=-2b$;
  \item An $(L,H,G)$ is $**$-scalar flat w.r.t. $J_1$ and $J_2$ if and only if $a=b^{-1}-b$, $(b\neq0)$;
  \item Every $(L,H,G)$ has a positive $**$-scalar curvature w.r.t. $J_3$;
  \item An $(L,H,G)$ has positive (resp., negative) basic holomorphic sectional curvatures w.r.t. $J_1$ and $J_2$ if and only if $a>0$ and $b>1$ (resp., $a<0$ and $0<b<1$);
  \item Every $(L,H,G)$ has a positive basic holomorphic sectional curvatures w.r.t. $J_3$;
  \item An $(L,H,G)$ has positive basic totally real sectional curvatures w.r.t. $J_1$ and $J_2$ if and only if $a>0$ and $b>1$;
  \item An $(L,H,G)$ has positive (resp., negative) basic totally real sectional curvatures w.r.t. $J_3$ if and only if $a>0$ and $b>1$ (resp., $a<0$ and $0<b<1$).
\end{enumerate}
\end{thm}
\begin{proof}
By virtue of \eqref{res-46}, we establish the truthfulness of the statement.
\end{proof}

\end{document}